\documentclass[12pt]{article}

\usepackage{graphicx}
\usepackage{latexsym}
\usepackage{array,fancybox}
\usepackage{amsmath,amssymb}
\usepackage{amsthm}
\usepackage{hyperref}
\usepackage[autostyle]{csquotes}

\numberwithin{equation}{section}
 \def\R{{\ifmmode{\rm I}\mkern-4mu{\rm R}
    \else\leavevmode\hbox{I}\kern-.17em\hbox{R}\fi}}
\def\N{{\ifmmode{\rm I}\mkern-3.5mu{\rm N}
    \else\leavevmode\hbox{I}\kern-.16em \hbox{N}\fi} }
\def\C{\ifmmode{{\rm C}\mkern-15mu{\phantom{\rm t}\vrule}}\mkern10mu
    \else\leavevmode\hbox{C}\kern-.5em\hbox{I}\kern.3em\fi}
\def\Ch{\ifmmode{\hat{\rm C}\mkern-15mu{\phantom{\rm t}\vrule}}\mkern10mu
    \else\leavevmode\hbox{C}\kern-.5em\hbox{I}\kern.3em\fi}
\def\Q{\ifmmode{{\rm Q}\mkern-16mu{\phantom{\rm t}\vrule}}\mkern10mu
    \else\leavevmode\hbox{Q}\kern-.57em\hbox{I}\kern.3em\fi}
\def\F{{\ifmmode{\rm I}\mkern-3.5mu{\rm F}
    \else\leavevmode\hbox{I}\kern-.16em \hbox{F}\fi} }

\newtheorem{theo}{Theorem}[section]
\newtheorem{prop}{Proposition}[section]
\newtheorem{lemm}{Lemma}[section]
\theoremstyle{definition}
\newtheorem{defi}{Definition}[section]

\begin{document}
  
                   \begin{center}
   {\large \bf Large-time behavior of solutions to a thermo-diffusion system with Smoluchowski interactions
 }
\\[1cm]
                   \end{center}
 
                   \begin{center}
                 {\sc Toyohiko Aiki}\\
 Department of Mathematical and Physical Sciences, Faculty of Science, \\
 Japan Women's University\\
 2-8-1 Mejirodai, Bunkyo-ku, Tokyo, 112-8681 Japan\\
 (aikit@fc.jwu.ac.jp)
                   \vspace{0.2cm}
                   
                {\sc Adrian Muntean}\\
 Department of Mathematics and Computer Science, Karlstad University \\
SE-651 88 Karlstad, Sweden \\
(adrian.muntean@kau.se)

       \end{center}
       \vspace{1cm}
        
\hspace*{-0.6cm}{\bf Abstract.} 
We prove the large time behavior of solutions to a coupled thermo-diffusion arising in the modelling of the motion of hot colloidal particles in porous media. Additionally, we also ensure the uniqueness of solutions of the target problem. The main mathematical difficulty is due to the presence in the right-hand side of the equations of products between temperature and concentration gradients. Such terms mimic the so-called thermodynamic Soret and Dufour effects. These are cross-coupling terms  emphasizing in this context a strong interplay between heat conduction and molecular diffusion.

\hspace*{-0.6cm}{\bf Keywords.} 
Thermo-diffusion; gradient estimates; large-time behavior; Sorret and Dufour effects
 
 \hspace*{-0.6cm}{\bf MSC 2010:} 35Q79; 35K55; 35B45; 35B40 
\section{Introduction} 

Populations of colloids can be driven into motion by gradients in chemical, electrostatic, or thermal fields that may exist externally to the colloids; see  \cite{Ramin,Herz}, e.g.  This paper is concerned with the mathematical analysis of a scenario involving the joint effect of gradients in chemical and thermal fields that we refer to here as thermo-diffusion. Particularly, we study the large-time behavior of the following class of thermo-diffusion systems -- a nonlinear coupled system of partial differential equations with homogeneous Neumann boundary conditions described as follows:

Let $\Omega \subset {\mathbb R}^3$ with smooth boundary $\Gamma := \partial \Omega$, and $Q(T) =   (0,T) \times \Omega$ and $S(T) := (0,T) \times \Gamma$ for $T > 0$. 
The problem is to find a pair of functions $(\theta,u)$, with  $u = (u_1, u_2, \cdots, u_N)$, satisfying 
\begin{align}
	& \theta_t - \kappa \Delta \theta - \tau \sum_{i=1}^N \nabla^{\delta_0} u_i \cdot \nabla \theta =0 
                   \mbox{ in } Q(T),  \label{eq1} \\
	& u_{it} - \kappa_i \Delta u_{i} - \tau_i \nabla \theta \cdot \nabla u_i =  R_i(u)  \mbox{ in } Q(T) 
       \mbox{ for each } i, \label{eq2} \\
	& - \kappa \nabla\theta \cdot \nu =0, - \kappa_i \nabla u_i \cdot \nu =0 \mbox{ for each } i 
 \mbox{ on } S(T),  \label{bc} \\
	& \theta(0, x) = \theta_0 (x), u_i(0,x) = u_{0i}(x)  \mbox{ for each } i  \mbox{ on } \Omega, \label{ic}
\end{align}
where $\kappa$ and $ \kappa_i ( i = 1,2 \cdots, N)$  are the diffusion constants, $\delta_0$, 
 $\tau$ and $ \tau_i ( i = 1,2 \cdots, N)$ are positive constants, $\nu$ is the outward normal vector to $\Gamma$, and $R_i: {\mathbb R}^N \to {\mathbb R}$ is given as 
$$ R_i(u) = \frac{1}{2} \sum_{k+j = i} \beta_{kj} u_k^+ u_j^+ - 
\sum_{j=1}\beta_{ij} u_i^+ u_j^+ \mbox{ for each } i, $$
where $r^+$ indicates its positive part for $r \in {\mathbb R}$,  
 $\beta_{kj}$ are positive constants (discrete values of aggregation and fragmentation kernels) such that $\beta_{kj} = \beta_{jk}$ ($j, k =1,2 \cdots, N$). 
Moreover, for a given choice of $\delta > 0$, we use following notation:
$$ J(x) = \left\{ \begin{array}{ll}
                C_m\exp(- \frac{1}{1 -|x|^2}) & \mbox{ if } |x| <1, \\
                 0 & \mbox{ otherwise}, 
           \end{array} \right. $$
where $C_m$ is a positive constant chosen such that $\int_{{\mathbb R}^3} J(x) dx = 1$.
Here, we put $J_{\delta}(x) = \delta^{-3} J(x/\delta)$ and,  for a measurable function $f$ on $\Omega$, we employ 
$$ \nabla^{\delta} f = \nabla (J_{\delta} \ast f ) = 
\nabla (\int_{{\mathbb R}^3} J_{\delta}(x - y) f(y) dy).  $$

We denote by P the above system (\ref{eq1}) $\sim$ (\ref{ic}). 

The pair $(\theta,u)$ refers to the  unknowns in the system, i.e.  $\theta$ is the temperature field, while  $u = (u_1, u_2, \cdots, u_N)$ is the vector of $N$ interacting colloidal populations.  The reaction term $R(\cdot)$ models the classical Smoluchovski interaction production (see, for instance, chapter 2 in \cite{Krehel}).

The structure of this system has been proposed in \cite{Krehel} as a mathematical model supposed to describe simultaneous effects   between heat conduction and moelcular diffusion arising when populations of hot colloids like to "diffuse" inside porous materials. The process is usually called {\em thermo-diffusion} and considerable phenomenological understanding is available (compare \cite{Mazur} or the more recent accounts by Wojnar \cite{Wojnar1,Wojnar2}). Regarding the presence in the right-hand side of the model equaations of the products between temperature and concentration gradients -- mimicking thermodynamic Soret and Dufour effects pointing out a strong interplay between heat conduction and molecular diffusion --  we refer the reader to \cite{Oleh-A-M,Oleh-A-K,Sina}. In these settings, such strongly nonlinear  structures arising in the model equations play a decisive role in capturing the expected evolution of the physical system. 

In the framework of this paper, we are interested in understanding the large time behavior of a given porous material exposed to thermo-diffusive infiltrations, very much in the spirit of related mathematical work done for a conceptually different problem referring to the chemical corrosion of concrete; see e.g. \cite{Toyohiko, Kota}.  

The major mathematical difficulty encountered here is the presence of nonlinear terms of the type $\nabla \theta \cdot \nabla u_i$.   A careful look at our estimates will discover that the presence of terms like $\nabla^{\delta_0} u_i \cdot \nabla \theta$ is essential to ensure ultimately a good (time independent) control on the $L^\infty$ bounds on the gradients of both temperature and colloidal concentrations. The regularization parameter $\delta_0$ arising in  $\nabla^{\delta_0} u_i$ can be removed only in one space dimension using a suitable combination of compactness arguments for strong solutions to problem P and the Leray-Schauder fixed point principle; see for instance the line of thought in \cite{Benes}. 

Using a couple of approximating problems and a suitable grip on the gradient of concentrations and of the gradient of temperature, we prove that, for sufficiently large time, all transport terms in problem P disappear, the limiting evolution of the concentration being simply governed by the ordinary differential equations governing the Smoluchowski dynamics.

\section{Main result}

We begin with the definition of our concept of solution to problem P. 

\begin{defi}
Let $\theta$ and $u_i (i=1,2, \cdots, N)$ be  functions on $Q (T) $ for $T > 0$ and 
$u = (u_1, u_2, \cdots, u_N)$. 
We call that a pair $\{\theta, u\}$ is a solution of P on $[0, T]$ if the conditions (S1) and (S2) hold: 
\begin{itemize}
	\item[(S1)] $(\theta, u) \in X(T)^{N+1}$, where $X(T) = L^{\infty}(Q(T)) \cap W^{1,2}(0,T; L^2(\Omega)) \cap L^{\infty}(0,T; H^1(\Omega)) \cap L^2(0,T; H^2(\Omega))$. 
	\item[(S2)]  (\ref{eq1}) $\sim$ (\ref{ic}) hold in the usual sense.  
\end{itemize}
Moreover, we say that  $\{\theta, u\}$ is a solution of P on $[0, \infty)$, 
if it is a solution of P on $[0, T]$ for any $T >0$.  
\end{defi}

For simplicity, we put 
$$ L^{\infty}_+(\Omega) = \{z \in L^{\infty}(\Omega): z \geq 0 \mbox{ a.e. on } \Omega\},  $$
and write 
$$ |u|_{L^2(\Omega)} :=  |u|_{L^2(\Omega)^N},   
 |\nabla \theta|_{L^2(\Omega)} :=  |\nabla \theta|_{L^2(\Omega)^3} \mbox{ and so on. } $$

The first theorem of this paper guarantees the existence and the large time behavior of the problem P. 
\begin{theo} \label{th1}
If $\theta_0 \in H^1(\Omega) \cap L^{\infty}_+(\Omega)$ and  
$u_{0i} \in H^1(\Omega) \cap L^{\infty}_+(\Omega)$ for each $i$, then P has a solution on $[0, \infty)$.
Moreover,   for each $i=1,2 \cdots, N$ we have $u_i(t) \to 0$ as $t \to \infty$. More precisely, 
there exists a function $y_i \in L^{\infty}(0, \infty) \cap L^2(0,\infty)$ such that 
\begin{equation}
 0 \leq u_i(t,x) \leq y_i(t) \quad \mbox{ for a.e. } x \in \Omega \mbox{ and } t > 0. 
\label{assert2}
\end{equation} 
\end{theo}

The second theorem is concerned with the uniqueness of solutions to probelm P. 
\begin{theo} \label{th2}
(1) Under the same assumptions as in Theorem \ref{th1} if 
$\theta_0 \in W^{1, \infty}(\Omega)$ and  
$u_{0i} \in W^{1,\infty}(\Omega)$ for each $i$, then there exists a  solution $\{\theta, u\}$ of P  satisfying 
$\theta \in L^{\infty}(0,T; W^{1,\infty}(\Omega))$ and 
$u_i \in L^{\infty}(0,T; W^{1,\infty}(\Omega))$ for $i$. 

(2) Let $\{\theta^{(k)}, u^{(k)}\}$ be a solution of P for $k = 1, 2$. 
If  $\theta^{(k)} \in L^{\infty}(0,T; W^{1,\infty}(\Omega))$ and 
$u_i^{(k)} \in L^{\infty}(0,T; W^{1,\infty}(\Omega))$, $k = 1, 2$ and each $i$, then 
$\theta^{(1)} = \theta^{(2)}$ and  $u^{(1)} = u^{(2)}$ a.e on $Q(T)$.   
\end{theo}

The remainder of the paper is concerned with proving these two results. 

To prove Theorem \ref{th1} we consider the approximation problem 
P$_{\varepsilon} :=$ \\ $  \{(\ref{eq1}), (\ref{eqa2}), (\ref{bc}), (\ref{ic})\}$ of the problem P 
for $\varepsilon >0$: 
\begin{equation}
u_{it} - \kappa_i \Delta u_{i} - \tau_i \nabla^{\varepsilon} \theta \cdot \nabla u_i =  R_i(u)  \mbox{ in } Q(T) 
       \mbox{ for each } i, 
\label{eqa2}
\end{equation} 
Moreover, we approximate P$_{\varepsilon}$ by P$_{\varepsilon, n} :=$ $\{(\ref{eq1}), (\ref{eqam2}), (\ref{bc}), (\ref{ic})\}$  for  $\varepsilon >0$ and  $n >0$: 
\begin{equation}
u_{it} - \kappa_i \Delta u_{i} - \tau_i \nabla^{\varepsilon} \theta \cdot \nabla u_i =  R_{in}(u)  \mbox{ in } Q(T)        \mbox{ for each } i, 
\label{eqam2}
\end{equation}
where $R_{in}(s_1, s_2, \cdots, s_N) := R_i(\sigma_n(s_1), \sigma_n(s_2), \cdots, \sigma_n(s_N))$ and 
$$
\sigma_n(r) = \left\{ \begin{array}{ll} 
      n & \mbox{ if } r > n, \\
      r & \mbox{ if }  0 \leq r \leq n, \\
      0 & \mbox{ otherwise,  }  \end{array} \right. 
          \quad \mbox{ for } r \in {\mathbb R}. 
$$

\vskip 12pt
In Section \ref{sec3} we shall show the existence of a solutions to both P$_{\varepsilon}$ and P$_{\varepsilon, n}$ for $\varepsilon > 0$ and $n > 0$. We give some uniform estimates for solutions of P$_{\varepsilon}$ with respect to $\varepsilon$ in Section \ref{sec4}. 
Finally, after controlling in terms of uniform estimates the solutions to the auxiliary problems, we give in Section \ref{final} the proofs of Theorems \ref{th1} and \ref{th2}. 

Throughout this paper we assume that the boundary $\Omega$ is sufficiently smooth such that 
\begin{equation}
 |f|_{H^2(\Omega)} \leq C_{\Omega} ( |\Delta f|_{L^2(\Omega)}  + |f|_{H^1(\Omega)}) 
        \mbox{ for } f \in H^2(\Omega) \mbox{ with } \nabla f \cdot \nu = 0 \mbox{ on } \Gamma,  \label{LU}
\end{equation}
where $C_{\Omega}$ is a positive constant (see Theorem 25.3 in Chapter of \cite{Neumann-H2}).

Finally, we list here a couple of very useful inequalities concerned with $\nabla^{\varepsilon}$ 
and $J_{\varepsilon}$ (see,  for example,   \cite{Adams}). 
For all $1 \leq p \leq \infty$, $q > 1$ and $\varepsilon  > 0$ it holds that 
\begin{align}
& |\nabla^{\varepsilon} f |_{L^p(\Omega)} \leq c_{p,\varepsilon} |f|_{L^2(\Omega)} \quad \mbox{ for }
          f \in L^2(\Omega),  \label{moli1} \\
& |\nabla^{\varepsilon} f |_{L^q(\Omega)} \leq
 c_{q} |\nabla f|_{L^q(\Omega)} \quad \mbox{ for }
          f \in W^{1,q}(\Omega), \label{moli2} \\
& |J_{\varepsilon}*f|_{L^2(\Omega)} \leq |f|_{L^2(\Omega)} \quad \mbox{ for }
          f \in L^2(\Omega), \label{moli3} \end{align}
where $c_{p,\varepsilon}$ and $c_q$ are positive constants.

\section{Approximate  problems} \label{sec3}
The aim of this section is to provide the following proposition.

\begin{prop} \label{pro1}
Let $\varepsilon > 0$, $T> 0$ and $n > 0$. If $\theta_0 \in H^1(\Omega) \cap L_+^{\infty}(\Omega)$ and  $u_{0i} \in H^1(\Omega) \cap L_+^{\infty}(\Omega)$ for each $i$, 
then P$_{\varepsilon,n}$  has a unique solution on $[0, T]$. 
\end{prop}

The proof is similar to the proof of \cite[Theorem 3.8]{Oleh-A-M}. 
However, to establish the large-time behavior of the solution we need more precise estimates.
Then we give the proof of Proposition \ref{pro1}, here. 

First, for a given function $\hat{u} = (\hat{u}_1, \hat{u}_2, \cdots, \hat{u}_N)$   
on $Q(T)$ we consider the following problem LP($\hat{u}$): 
\begin{align}
	& \theta_t - \kappa \Delta \theta - \tau \sum_{i=1}^N \nabla^{\delta_0} \hat{u}_i \cdot \nabla \theta =0 
                   \mbox{ in } Q(T),  \label{lp1} \\
	& u_{it} - \kappa_i \Delta u_{i} - \tau_i \nabla^{\varepsilon} \theta \cdot \nabla u_i =  
   R_{in}(\hat{u})  \mbox{ in } Q(T) 
       \mbox{ for each } i, \label{lp2} \\
	& - \kappa \nabla\theta \cdot \nu =0, - \kappa_i \nabla u_i \cdot \nu =0 \mbox{ for each } i 
 \mbox{ on } S(T),  \label{lpbc} \\
	& \theta(0, x) = \theta_0 (x), u_i(0,x) = u_{0i}(x)  \mbox{ for each } i  \mbox{ on } \Omega. \label{lpic}
\end{align}

Since this problem  LP($\hat{\theta}, \hat{u}$) is linear, we can easily get:
\begin{lemm} \label{lem3-2}
Let $\varepsilon > 0$, $n > 0$ and $T> 0$. If $\theta_0 \in H^1(\Omega)$, $u_{0i} \in H^1(\Omega)$  and $\hat{u}_i \in L^2(0, T; L^2(\Omega))$
for each $i$, then  LP($\hat{u}$)  has a unique solution $(\theta, u) \in X(T)^{N+1}$ on $[0, T]$. 
\end{lemm}

For the proof of Proposition \ref{pro1}, we give two lemmas. 
\begin{lemm} \label{lem3-3}
Let $\varepsilon > 0$, $n > 0$ and $T> 0$. If $\theta_0 \in H^1(\Omega) \cap L_+^{\infty}(\Omega)$,  $u_{0i} \in H^1(\Omega)$  and $\hat{u}_i \in L^2(0, T; L^2(\Omega))$
for each $i$, then the solution $\{\theta, u\}$ of  LP($\hat{u}$) satisfies 
$$ 0 \leq \theta \leq |\theta_0|_{L^{\infty}(\Omega)} \quad \mbox{ a.e. on } Q(T). $$
\end{lemm}
\begin{proof}
We put $\tilde{\theta} = -\theta$. Then, it holds that   
\begin{equation}
 \tilde{\theta}_t - \kappa \Delta \tilde{\theta} = 
\tau \sum_{i=1}^N \nabla^{\delta_0} \hat{u}_i \cdot \nabla \tilde{\theta} 
\quad \mbox{ a.e. on } Q(T).  \label{3-5}
\end{equation}
Here, by multiplying (\ref{3-5}) by $[\tilde{\theta}]^+$ and  then integrating it over $\Omega$ 
we obtain 
\begin{align*}
& \frac{1}{2} \frac{d}{dt} | [\tilde{\theta}]^+|_{L^2(\Omega)}^2 
 + \kappa  | \nabla [\tilde{\theta}]^+|_{L^2(\Omega)}^2 \\
= \ & \tau \sum_{i=1}^N   \int_{\Omega} ( \nabla^{\delta_0} \hat{u}_i \cdot \nabla \tilde{\theta})   
        [\tilde{\theta}]^+ dx  \\
\leq \ & \tau \sum_{i=1}^N | \nabla^{\delta_0} \hat{u}_i|_{L^{\infty}(Q(T))} 
            |\nabla [\tilde{\theta}]^+|_{L^2(\Omega)}  | [\tilde{\theta}]^+|_{L^2(\Omega)}  
   \quad \mbox{ a.e. on } [0,T]
\end{align*}
so that 
\begin{align}
 \frac{1}{2} \frac{d}{dt} | [\tilde{\theta}]^+|_{L^2(\Omega)}^2 
 + \frac{\kappa}{2}   | \nabla [\tilde{\theta}]^+|_{L^2(\Omega)}^2 
\leq   \frac{\tau^2 \sqrt{N}}{2 \kappa}  \sum_{i=1}^N | \nabla^{\delta_0} \hat{u}_i|_{L^{\infty}(Q(T))}^2 
             | [\tilde{\theta}]^+|_{L^2(\Omega)}^2   \mbox{ a.e. on } [0,T].  \label{3-6}
\end{align}
By  applying Gronwall's inequality to the above inequality, we get
$  | [\tilde{\theta}]^+(t)|_{L^2(\Omega)}^2 = 0$ for a.e $t \in[0,T]$, namely, 
$\theta \geq 0$ a.e. on $Q(T)$. 
 
Next, we put $M_0 = |\theta_0|_{L^{\infty}(\Omega)}$ and multiply  (\ref{lp1}) by 
$[\theta - M_0]^+$. Similarly to (\ref{3-6}), we can get 
\begin{align*}
 &  \frac{1}{2} \frac{d}{dt} | [\theta - M_0]^+|_{L^2(\Omega)}^2 
 + \frac{\kappa}{2}   | \nabla [\theta - M_0]^+|_{L^2(\Omega)}^2  \\
\leq \  &   \frac{\tau^2 \sqrt{N}}{2 \kappa}  \sum_{i=1}^N | \nabla^{\delta_0} \hat{u}_i|_{L^{\infty}(Q(T))} 
             | [\theta - M_0]^+|_{L^2(\Omega)}^2 \quad   \mbox{ a.e. on } [0,T].  
\end{align*}
The use of Gronwall's inequality completes the proof of this lemma. 
\end{proof}

\begin{lemm} \label{lem3-4}
Under the same assumptions as in Lemma \ref{lem3-3},  let 
$\{\theta, u\}$ be a solution of  LP($\hat{u}$) on $[0,T]$. 
Then there exists a positive constant $C_{n,\varepsilon}$ such that 
$$ |u_{i}(t)|_{L^2(\Omega)} \leq C_{n, \varepsilon} \quad \mbox{ for } 0 \leq t \leq T. 
$$
\end{lemm}
\begin{proof}
For simplicity, we put $C_{n}^{(1)} = \max\{ R_{in}(s) | s \in {\mathbb R}, i = 1, 2, \cdots, N\}$. 
By multiplying (\ref{lp2}) by $u_i$ and integrating it, we show that 
\begin{align*}
& \frac{1}{2} \frac{d}{dt} |u_i|_{L^2(\Omega)}^2 + \kappa_i |\nabla u_i|_{L^2(\Omega)}^2 \\
=  &  \tau_i \int_{\Omega} ( \nabla^{\varepsilon}  \theta  \cdot \nabla u_i) u_i dx 
+ \int_{\Omega} R_{in}(\hat{u})  u_i dx   \\
\leq & \frac{\kappa_i}{2} |\nabla u_i|_{L^2(\Omega)}^2 
   + \frac{\tau_i^2}{2 \kappa_i} |\nabla^{\varepsilon} \theta|_{L^{\infty}(\Omega)}^2 
|u_i|_{L^2(\Omega)}^2 
 + C_n^{(1)} \int_{\Omega} |u_i| dx \mbox{ a.e. on } [0,T]. 
\end{align*}
Here, by (\ref{moli1}) and Lemma \ref{lem3-3} we have 
$ |\nabla^{\varepsilon} \theta|_{L^{\infty}(\Omega)}  \leq c_{\infty,\varepsilon} |\theta_0|_{L^{\infty}(\Omega)} := M_{1\varepsilon}$ a.e. on $[0,T]$. It is easy to see that 
\begin{align*}
 \frac{1}{2} \frac{d}{dt} |u_i|_{L^2(\Omega)}^2 + \frac{\kappa_i}{2}  |\nabla u_i|_{L^2(\Omega)}^2
\leq   
    \frac{\tau_i^2}{2 \kappa_i} M_{1\varepsilon}^2 |u_i|_{L^2(\Omega)}^2 
 + C_n^{(1)} ( |u_i|_{L^2(\Omega)}^2 + |\Omega|)   \mbox{ a.e. on } [0,T], 
\end{align*}
where $|\Omega|  = \int_{\Omega}dx$. Then, Gronwall's inequality implies the assertion of this lemma. 
\end{proof}

By combining suitably these two lemmas we are able to prove Proposition \ref{pro1}. 

\begin{proof}[Proof of Proposition \ref{pro1}]
We define the operator $\Lambda: L^2(0,T; L^2(\Omega))^{N} \to L^2(0,T; L^2(\Omega))^{N} $ 
by $\Lambda(\hat{u}) = u$ for $\hat{u}  \in L^2(0,T; L^2(\Omega))^{N}$,  
where $\{\theta, u\}$ is a solution of   LP($\hat{u}$).
 Moreover, for $M > 0$, we take 
$$ K_M(T) := \{ u \in L^2(0,T; L^2(\Omega))^{N}| 
    \sum_{i=1}^N \int_0^T |u_i|_{L^2(\Omega)}^2 dt \leq M\}. 
$$

For any $\hat{u} \in K_M(T)$, let $\Lambda(\hat{u}) = (\theta, u)$. 
By Lemma \ref{lem3-3} and Lemma \ref{lem3-4},  we obtain 
$$ 
 \sum_{i=1}^N \int_0^T |u_i|_{L^2(\Omega)}^2 dt 
\leq  C_{n, \varepsilon}^2 NT := M_2. $$
Obviously, for $M \geq M_2$ $\Lambda: K_M(T) \to K_M(T)$. 
Moreover, we multiply (\ref{lp1}) by $\theta$, and then, in a similar way as in the proof of Lemma \ref{lem3-4},  
we can show that  
\begin{align*}
& \frac{1}{2} \frac{d}{dt} |\theta|_{L^2(\Omega)}^2 + \kappa |\nabla \theta|_{L^2(\Omega)}^2 \\
= \  &  \tau \sum_{i=1}^N \int_{\Omega} ( \nabla^{\delta_0}  \hat{u}_i  \cdot \nabla \theta) \theta dx  \\
\leq \ & \frac{\kappa}{2} |\nabla \theta|_{L^2(\Omega)}^2  + 
  \frac{\tau^2 M_0^2 N}{2\kappa} \sum_{i=1}^N |\nabla^{\delta_0} \hat{u}_i|_{L^2(\Omega)}^2  \\
\leq \ & \frac{\kappa}{2} |\nabla \theta|_{L^2(\Omega)}^2  + 
  \frac{\tau^2 M_0^2 N}{2\kappa} c_{2,\delta_0}
     \sum_{i=1}^N |\hat{u}_i|_{L^2(\Omega)}^2   \quad \mbox{ a.e. on } [0,T], 
\end{align*}
where $M_0 = |\theta_0|_{L^{\infty}(\Omega)}$ and $c_{2,\delta_0}$ is the positive constant defined in (\ref{moli1}). By integrating it we can obtain a positive constant $M_3$ such that 
$$ \int_0^T |\nabla \theta|_{L^2(\Omega)}^2 dt \leq M_3 \mbox{ for any } 
  \hat{u} \in K_M(T). $$

Next, we show that we can take $T > 0$ such that $\Lambda$ is a contraction mapping. 
For $\hat{u}^{(1)}$,  $\hat{u}^{(2)} \in K_M(T)$ let 
$(\theta^{(j)}, u^{(j)}) = \Lambda(\hat{u}^{(j)})$ ($j = 1,2$), 
$\theta = \theta^{(1)} - \theta^{(2)}$, $u = u^{(1)} - u^{(2)}$ and 
$\hat{u} = \hat{u}^{(1)} - \hat{u}^{(2)}$. 
Then it holds that 
\begin{align}
& \theta_t - \kappa \Delta \theta = \tau \sum_{i=1}^N
  (\nabla^{\delta_0} \hat{u}_i^{(1)} \cdot \nabla \theta^{(1)} - 
           \nabla^{\delta_0} \hat{u}_i^{(2)} \cdot \nabla \theta^{(2)}) \mbox{ a.e. on } Q(T), 
  \label{3-7} \\
& u_{it} - \kappa_i \Delta u_i = \tau_i (\nabla^{\varepsilon } {\theta}^{(1)} \cdot \nabla u_i^{(1)} - 
           \nabla^{\varepsilon} \theta^{(2)} \cdot \nabla u_i^{(2)})  
               +R_{in}(\hat{u}^{(1)}) - R_{in}(\hat{u}^{(2)}) 
 \mbox{ a.e. on } Q(T)  \label{3-8}
\end{align}
for each $i$. By multiplying (\ref{3-7}) by $\theta$ and integrating it, we see that 
\begin{align*} 
&  \frac{1}{2} \frac{d}{dt} |\theta|_{L^2(\Omega)} + \kappa |\nabla \theta|_{L^2(\Omega)}^2 \\
= \ & \tau \sum_{i=1}^N (
    \int_{\Omega} (\nabla^{\delta_0} \hat{u}_i \cdot \nabla \theta^{(1)}) \theta dx 
+     \int_{\Omega} (\nabla^{\delta_0} \hat{u}_i^{(2)} \cdot \nabla \theta) \theta dx)  \\
\leq \ &  \tau \sum_{i=1}^N (
    |\nabla^{\delta_0} \hat{u}_i|_{L^{\infty}(\Omega)} |\nabla \theta^{(1)}|_{L^2(\Omega)} 
               |\theta|_{L^2(\Omega)} 
+  |\nabla^{\delta_0} \hat{u}_i^{(2)}|_{L^{\infty}(\Omega)} |\nabla \theta|_{L^2(\Omega)} 
     |\theta|_{L^2(\Omega)} ) \\
\leq \ &  \tau c_{2,\delta_0} \sum_{i=1}^N (
    |\hat{u}_i|_{L^{2}(\Omega)} |\nabla \theta^{(1)}|_{L^2(\Omega)} 
               |\theta|_{L^2(\Omega)} 
+  |\hat{u}_i^{(2)}|_{L^{2}(\Omega)} |\nabla \theta|_{L^2(\Omega)} 
     |\theta|_{L^2(\Omega)} ) \\
\leq \ & \frac{\tau c_{2,\delta_0}}{2}  \sum_{i=1}^N (
    |\hat{u}_i|_{L^{2}(\Omega)}^2 +  |\nabla \theta^{(1)}|_{L^2(\Omega)}^2  |\theta|_{L^2(\Omega)}^2) 
+ \frac{\kappa}{2}  |\nabla \theta|_{L^2(\Omega)}^2  \\
& + \frac{\tau^2 c_{2,\delta_0}^2}{2\kappa}   (\sum_{i=1}^N |\hat{u}_i^{(2)}|_{L^{2}(\Omega)})^2
     |\theta|_{L^2(\Omega)}^2 \quad  \mbox{ a.e. on }  [0,T]
\end{align*}
so  that 
\begin{align} 
&  \frac{1}{2} \frac{d}{dt} |\theta|_{L^2(\Omega)} + \frac{\kappa}{2} |\nabla \theta|_{L^2(\Omega)}^2 
  \nonumber \\
\leq \ & \frac{\tau c_{2,\delta_0}}{2}  \sum_{i=1}^N   |\hat{u}_i|_{L^{2}(\Omega)}^2 
 + \left( \frac{\tau c_{2,\delta_0}}{2} N  |\nabla \theta^{(1)}|_{L^2(\Omega)}^2 
   +  \frac{\tau^2 c_{2,\delta_0}^2}{2\kappa} N |\hat{u}^{(2)}|_{L^{2}(\Omega)}^2
 \right) |\theta|_{L^2(\Omega)}^2  \label{3-9}
\end{align}
a.e. on $[0,T]$. Similarly, it follows from (\ref{3-8}) that 
\begin{align} 
&  \frac{1}{2} \frac{d}{dt} |u_i|_{L^2(\Omega)} + \frac{\kappa_i}{2} |\nabla u_i|_{L^2(\Omega)}^2 
  \nonumber \\
\leq \ & \frac{\tau_i c_{2,\varepsilon}}{2}  |\theta|_{L^{2}(\Omega)}^2 
  + C_n^{(2)} \sum_{i=2}^N |\hat{u}_i|_{L^2(\Omega)}^2  \label{3-10} \\
& \  + \left( \frac{\tau_i}{2} N  |\nabla u_i^{(1)}|_{L^2(\Omega)}^2 
  +  \frac{\tau^2 c_{2,\varepsilon}^2}{2\kappa_i} |\theta^{(2)}|_{L^{2}(\Omega)}^2 
   +  C_n^{(2)}  \right) |\theta|_{L^2(\Omega)}^2 
  \mbox{ a.e. on } [0,T] \mbox{ for each } i.   \nonumber
\end{align}
By adding (\ref{3-9}) and (\ref{3-10}), we have 
\begin{align*}
\frac{d}{dt} E(t) \leq C (|\hat{u}|_{L^2(\Omega)}^2 + F(t) E(t))  \quad 
             \mbox{ for a.e. } t \in [0,T],  
\end{align*}
where $E(t) = |\theta(t)|_{L^2(\Omega)}^2 +  |u(t)|_{L^2(\Omega)}^2$ and 
$F(t) =  |\nabla \theta^{(1)}|_{L^2(\Omega)}^2 +   |\hat{u}^{(2)}|_{L^{2}(\Omega)}^2
 +  |\nabla u^{(1)}|_{L^2(\Omega)}^2 +  |\theta^{(2)}|_{L^{2}(\Omega)}^2 + 1$ for 
a.e. $t \in [0,T]$ and $C$ is a suitable positive constant. 
Then, Gronwall's inequality implies that 
$$ E(t) \leq C e^{\int_0^t F(\tau) d\tau} \int_0^t  |\hat{u}(t)|_{L^2(\Omega)}^2 d\tau
  \quad \mbox{ for } t \in [0,T]. $$

Hence, we obtain in a straightforward manner that   
$$ |u|_{L^2(0,T; L^2(\Omega))}^2 \leq C T   |\hat{u}|_{L^2(0,T; L^2(\Omega))}^2
 \quad \mbox{ for } \hat{u} \in K_M(T). 
 $$
This inequality guarantees the existence of small  $T_0 > 0$ such that 
$\Lambda : K_M(T_0) \to K_M(T_0)$ is contraction. 
Therefore, $\Lambda$ has a unique fixed point. Moreover, since the choice of $T_0$ is independent of the initial values, the assertion of this lemma is true. 
\end{proof}

\section{Uniform estimates for approximate solutions} \label{sec4}
In this section we give several  auxiliary lemmas dealing with the derivation of uniform estimates of solutions to $P_{\varepsilon,n}$.

\begin{lemm} \label{lem4-1}
Let $\varepsilon > 0$, $T> 0$ and $n > 0$,  $\theta_0 \in H^1(\Omega) \cap L_+^{\infty}(\Omega)$, 
 $u_{0i} \in H^1(\Omega) \cap L_+^{\infty}(\Omega)$ for each $i$ and
 $\{\theta_{\varepsilon n}, u_{\varepsilon n} \}$ be a solution of P$_{\varepsilon,n}$ on $[0, T]$. 
Then we have $ 0 \leq \theta_{\varepsilon n} \leq |\theta_0|_{L^{\infty}(\Omega)}$ 
a.e. on $Q(T)$ for  $\varepsilon > 0$ and $n > 0$.  
\end{lemm}

We can prove this lemma in a similar way to the proof of Lemma \ref{lem3-3} so that we omit its proof. 
The next lemma is one of keys in the proofs of Theorems \ref{th1} and \ref{th2}. 

\begin{lemm} \label{lem4-2}
If $\varepsilon > 0$, $T> 0$,  $\theta_0 \in H^1(\Omega) \cap L_+^{\infty}(\Omega)$, 
 $u_{0i} \in H^1(\Omega) \cap L_+^{\infty}(\Omega)$ for each $i$, 
then there exist a positive constant $n_0$ and $y_i \in  L^{\infty}(0,\infty) \cap L^{2}(0,\infty) \cap W^{1,1}(0,\infty) \cap C^1([0,\infty)) $,  $i = 1, 2, \cdots, N$, 
 such that for $n  \geq n_0$ the solution  $\{\theta_{\varepsilon n}, u_{\varepsilon n} \}$ of P$_{\varepsilon,n}$ on $[0, T]$ satisfies
\begin{equation}
 0 \leq u_{i\varepsilon n} \leq y_i(t)  \quad \mbox{ a.e. on } Q(T) \mbox{ for each } i, 
\label{4-0}
\end{equation}
where $u_{\varepsilon n} = (u_{1\varepsilon n}, u_{2\varepsilon n}, \cdots, u_{N\varepsilon n})$. 
Moreover, for each $\varepsilon > 0$ P$_{\varepsilon}$ has a solution $\{\theta_{\varepsilon}, u_{\varepsilon} \}$ on $[0, \infty)$. 
 \end{lemm}

\begin{proof}
First, we can show that $u_{i\varepsilon n} \geq 0$ a.e. on $Q(T)$ by multiplying (\ref{eqam2}) by 
$-[-u_{i\varepsilon n}]^+$ (see the proof of Lemma \ref{lem3-3}). 

Next, let $y_1$ be a solution of the following initial value problem for the ordinary differential equation:
\begin{equation}
y_1' = - \beta_{11} y_1^2 \mbox{ on } [0,\infty), y_1(0) = |u_{01}|_{L^{\infty}(\Omega)}. 
\label{ODE1}
\end{equation}
Then, from (\ref{eqam2}),  we see that 
\begin{align}
 u_{1\varepsilon n t} - y_{1t} - \kappa_1 \Delta(u_{1 \varepsilon n} - y_1)  
& = \tau_1 \nabla^{\varepsilon} \theta_{\varepsilon n} \cdot (u_{1 \varepsilon n} - y_1) 
  + R_{in} (u_{\varepsilon n} ) - \beta_{11} y_1^2  \nonumber  \\
& \leq  \tau_1 \nabla^{\varepsilon} \theta_{\varepsilon n} \cdot (u_{1 \varepsilon n} - y_1) 
  + \beta_{11} \sigma_n(u_1)^2 - \beta_{11} y_1^2  \mbox{ a.e. on } Q(T).  \label{4-2}
\end{align}
 Since $\sigma_n(u_1)^2 - \beta_{11} y_1^2 \leq 0$ a.e. on $Q(T)$ for 
$n \geq |y_1|_{L^{\infty}(0,\infty)}$, by multiplying  (\ref{4-2}) by $[u_1 -y_1]^+$, we have 
\begin{align*}
& \frac{1}{2} \frac{d}{dt} |[u_{1\varepsilon n} - y_1]^+|_{L^2(\Omega)}^2 
  \kappa_1  |\nabla [u_{1\varepsilon n} - y_1]^+|_{L^2(\Omega)}^2 \\
  \leq   & \ 
\tau_1 \int_{\Omega} (\nabla^{\varepsilon} \theta_{\varepsilon n} \cdot
\nabla [u_{1\varepsilon n} - y_1]^+)  [u_{1\varepsilon n} - y_1]^+ dx  \\
\leq  & \ \frac{\kappa_1}{2}  \int_{\Omega} |\nabla [u_{1\varepsilon n} - y_1]^+|^2 dx
+ |\nabla^{\varepsilon} \theta_{\varepsilon n}|_{L^{\infty}(Q(T))}^2
     | [u_{1\varepsilon n} - y_1]^+|_{L^2(\Omega)}^2 
         \mbox{ a.e. on } [0,T].  
\end{align*}
The Gronwall's inequality yields
\begin{equation}
u_{1\varepsilon n} \leq y_1 \quad \mbox{ a.e. on } Q(T). 
\label{4-2.5}
\end{equation}

Here, we assume that for $i = 1, 2, \cdots, i_0$ there exists 
$y_i \in  L^{\infty}(0,\infty) \cap L^{2}(0,\infty) \cap W^{1,1}(0,\infty) \cap C^1([0,\infty)) $  
satisfying that for $n \geq \max\{ |y_i|_{L^{\infty}(0,\infty)} | i = 1, 2,  \cdots, i_0\}$ 
$$ 0 \leq u_{i \varepsilon n} \leq y_i \quad \mbox{ a.e. on } Q(T). $$ 
Let $i = i_0 +1$ and $y_i$ be a solution of  
$$ y_{i}' = \frac{\beta_0}{2}  \sum_{k=1}^{i-1} y_k^2 - \beta_{ii} y_i^2 \mbox{ on } [0,\infty), y_i(0) = |u_{0i}|_{L^{\infty}(\Omega)}. $$

 Then, by choosing $n \geq \max\{ |y_i|_{L^{\infty}(0,\infty)} | i = 1, 2,  \cdots, i_0 + 1\}$ 
 it is easy to see that 
\begin{align*}
 u_{i\varepsilon n t} - y_{it} - \kappa_1 \Delta(u_{1 \varepsilon n} - y_1)  
& = \tau_1 \nabla^{\varepsilon} \theta_{\varepsilon n} \cdot (u_{1 \varepsilon n} - y_1) 
  + R_{in} (u_{\varepsilon n} ) - \beta_{ii} y_i^2  \nonumber  \\
& \leq  \tau_1 \nabla^{\varepsilon} \theta_{\varepsilon n} \cdot (u_{1 \varepsilon n} - y_1) 
  + \beta_{ii} \sigma_n(u_1)^2 - \beta_{ii} y_i^2   \nonumber  \\
& \leq  \tau_1 \nabla^{\varepsilon} \theta_{\varepsilon n} \cdot (u_{1 \varepsilon n} - y_1) 
  \quad  \mbox{ a.e. on } Q(T).  
\end{align*}
Similarly to (\ref{4-2.5}), we obtain (\ref{4-0}). 

Furthermore, by taking $n \geq   \max\{ |y_i|_{L^{\infty}(0,\infty)} | i = 1, 2,  \cdots, N\}$ 
we have 
$$  0 \leq u_{i\varepsilon n} \leq \max\{ |y_i|_{L^{\infty}(0,\infty)} | i = 1, 2,  \cdots, N\}
   \mbox{ a.e. on } Q(T) \mbox{ for any } T > 0 \mbox{ and } \varepsilon >0,  
$$
that is,  $R_{in} (u_{\varepsilon n}) =  R_{i} (u_{\varepsilon n})$ a.e on $(0,T) \times \Omega$
for $T > 0$. Thus we have proved this lemma. 
\end{proof}

\begin{lemm} \label{lem4-3}
Under the same assumption as in Lemma \ref{lem4-2} 
let $\{\theta_{\varepsilon}, u_{\varepsilon}\}$ be a solution of  P$_{\varepsilon}$
on $[0,T]$ for each $\varepsilon > 0$ and $T > 0$. 
Then there exists a positive constant $C_1$  such that
\begin{align}
&  |\theta_{\varepsilon}|_{W^{1,2}(0,T; L^2(\Omega) ) }  + 
|\theta_{\varepsilon}|_{L^{\infty}(0,T; H^1(\Omega) ) }  \leq C_1  \quad \mbox{ for } \varepsilon \in (0,1], 
       \label{4-5.1} 
  \\ 
& |u_{\varepsilon}|_{L^2(0,T; H^1(\Omega)) } \leq C_1  \quad 
  \mbox{ for } \varepsilon \in (0,1]. \label{4-6} 
\end{align}
\end{lemm}
\begin{proof}
Let $u_{\varepsilon} = (u_{1\varepsilon}, \cdots, u_{N\varepsilon})$ for $\varepsilon \in (0,1]$. 

First,  we multiply (\ref{eq1}) by $\theta_{\varepsilon t}$ and integrate it over $\Omega$. 
Then  on  account of (\ref{moli1}) we see that 
\begin{align}
&   |\theta_{\varepsilon t}|_{L^2(\Omega)}^2  
+ \frac{\kappa}{2} \frac{d}{dt} |\nabla \theta_{\varepsilon}|_{L^2(\Omega)}^2  
    \nonumber \\
= & \   \tau \sum_{i=1}^N \int_{\Omega}
      ( \nabla^{\delta_0} u_{i \varepsilon} \cdot \nabla \theta_{\varepsilon}) \theta_{\varepsilon t} dx 
             \nonumber \\
\leq & \  \frac{1}{2}  |\theta_{\varepsilon t}|_{L^2(\Omega)}^2 
 + \frac{ \tau^2}{2}  \sum_{i=1}^N 
      |\nabla^{\delta_0} u_{i \varepsilon}|_{L^{\infty}(\Omega)}^2 
      | \nabla \theta_{\varepsilon}|_{L^2(\Omega)} 
             \nonumber \\
\leq & \ 
   \frac{1}{2}  |\theta_{\varepsilon t}|_{L^2(\Omega)}^2 
+ \frac{ \tau^2  c_{\infty,\delta_0}^2}{2}      \sum_{i=1}^N 
      |u_{i \varepsilon}|_{L^2(\Omega)}^2 
      | \nabla \theta_{\varepsilon}|_{L^2(\Omega)}^2    \quad \mbox{ a.e. on } [0,T] 
    \mbox{ for } \varepsilon \in (0,1].          \nonumber    
\end{align}
Obviously, we get
\begin{equation}
 \frac{1}{2} |\theta_{\varepsilon t}|_{L^2(\Omega)}^2 
+ \frac{\kappa}{2} \frac{d}{dt} |\nabla \theta_{\varepsilon}|_{L^2(\Omega)}^2  
\leq 
\frac{ \tau^2  c_{\infty,\delta_0}^2}{2}     
     C_* |\Omega| N 
      | \nabla \theta_{\varepsilon}|_{L^2(\Omega)}^2  \mbox{ a.e. on } [0,T] 
    \mbox{ for } \varepsilon \in (0,1],   \label{4-7}
\end{equation} 
where  $C_* =  \max\{ |y_i|_{L^{\infty}(0,\infty)} | i = 1, 2,  \cdots, N\}$.

By applying Grownwall's inequality to (\ref{4-7}),   (\ref{4-0}) guarantees  (\ref{4-5.1}). 

Next, we multiply (\ref{eqa2}) by $u_{i \varepsilon}$ and integrate it over $\Omega$. 
Then  for each $i$ we get 
\begin{align}
&  \frac{1}{2} \frac{d}{dt} |u_{i \varepsilon}|_{L^2(\Omega)}^2  
  + \kappa_i \int_{\Omega} |\nabla u_{i \varepsilon} |^2 dx   \nonumber \\
= & \   \tau_i \int_{\Omega}
      ( \nabla^{\varepsilon} \theta_{\varepsilon} \cdot \nabla u_{i \varepsilon}) u_{i \varepsilon} dx 
   + \int_{\Omega} R_i(u_{\varepsilon}) u_{i \varepsilon} dx  
             \nonumber \\
\leq & \   \tau_i  
      |\nabla^{\varepsilon}  \theta_{\varepsilon}|_{L^2(\Omega)} 
      | \nabla u_{i \varepsilon}|_{L^2(\Omega)}  |u_{i \varepsilon}|_{L^{\infty}(\Omega)}  
  + \int_{\Omega} C_R C_* dx              \nonumber \\
\leq & \  c_{2}  \tau_i C_* \sqrt{|\Omega|}    
      | \nabla \theta_{\varepsilon}|_{L^2(\Omega)}    | \nabla u_{i \varepsilon}|_{L^2(\Omega)} 
 + C_R C_* |\Omega|        \nonumber \\
\leq & \  \frac{\kappa_i}{2} \int_{\Omega} |\nabla u_{i \varepsilon} |^2 dx
  + \frac{ c_{2}^2  \tau_i^2 C_*^2  |\Omega|}{2 \kappa_i}     
      | \nabla \theta_{\varepsilon}|_{L^2(\Omega)}^2 
 +  C_R C_* |\Omega|      
 \mbox{ a.e. on } [0,T]   \mbox{ for } \varepsilon \in (0,1],     
             \label{4-8}
\end{align}
where $C_R = \sup\{ R_i(u_{\varepsilon}) | \varepsilon \in (0,1], i = 1, 2, \cdots, N\}$. 

By (\ref{4-8}) and  (\ref{4-5.1})  we can get (\ref{4-6}). 
\end{proof}

\vskip 12pt
The following two lemmas are concerned with the essential estimates for the proof of existence part of Theorem 
\ref{th1}.

\begin{lemm} \label{lem4-4}
If  the same assumptions as in Lemma \ref{lem4-3}  hold, 
then the set $\{\nabla \theta_{\varepsilon} | \varepsilon \in (0,1] \}$ is bounded in 
$L^4(Q(T))^3$ for any $T > 0$. 
\end{lemm}

\begin{proof}
For $\varepsilon \in (0,1]$ and $T > 0$  let $\{\theta_{\varepsilon}, u_{\varepsilon}\}$ be a solution of 
 P$_{\varepsilon}$ on $[0,T]$,  and $u_{\varepsilon} = (u_{1\varepsilon}, \cdots, u_{N\varepsilon})$. 
Because of  $\theta_{\varepsilon}(t) \in H^2(\Omega)$ for a.e. $t \in [0,T]$
 Soblolev's embedding theorem implies that 
$\nabla \theta_{\varepsilon}(t) \in L^6(\Omega)^3$ for a.e. $t \in [0,T]$.  
Then we can multiply $\theta_{\varepsilon} |\nabla \theta_{\varepsilon}|^2$ to (\ref{eq1}) and get 
\begin{align}
& \int_{\Omega} \theta_{\varepsilon t } (\theta_{\varepsilon} |\nabla \theta_{\varepsilon}|^2) dx 
- \kappa \int_{\Omega} \Delta \theta_{\varepsilon} (\theta_{\varepsilon} |\nabla \theta_{\varepsilon}|^2) dx  \nonumber  \\
~= & \ \tau  \sum_{i=1}^N
\int_{\Omega} (\nabla^{\delta_0} u_{i \varepsilon} \cdot \nabla \theta_{\varepsilon})  
\theta_{\varepsilon} |\nabla \theta_{\varepsilon}|^2 dx  \quad \mbox{ a.e.  on }  [0,T]. \nonumber 
\end{align} 
We note that 
\begin{align}
 - \kappa \int_{\Omega} \Delta \theta_{\varepsilon} (\theta_{\varepsilon} |\nabla \theta_{\varepsilon}|^2) dx  \nonumber  
~= &   \kappa \int_{\Omega} \nabla \theta_{\varepsilon}  \cdot \nabla (\theta_{\varepsilon} |\nabla \theta_{\varepsilon}|^2) dx  \nonumber  \\
= &   \kappa \int_{\Omega} |\nabla \theta_{\varepsilon}|^4 dx 
+ 2  \kappa \int_{\Omega} \theta_{\varepsilon}  \sum_{i,j = 1}^3
 \frac{\partial^2  \theta_{\varepsilon}}{\partial x_i \partial x_j} 
 \frac{\partial  \theta_{\varepsilon}}{\partial x_i} 
 \frac{\partial  \theta_{\varepsilon}}{\partial x_j} dx  \nonumber  
\end{align} 
so that 
\begin{align}
 &  \  \kappa \int_{\Omega} |\nabla \theta_{\varepsilon}|^4 dx  \nonumber   \\
= & -  \int_{\Omega} \theta_{\varepsilon t } (\theta_{\varepsilon} |\nabla \theta_{\varepsilon}|^2) dx
 -2  \kappa \int_{\Omega} \theta_{\varepsilon}  \sum_{i,j = 1}^3
 \frac{\partial^2  \theta_{\varepsilon}}{\partial x_i \partial x_j} 
 \frac{\partial  \theta_{\varepsilon}}{\partial x_i} 
 \frac{\partial  \theta_{\varepsilon}}{\partial x_j} dx  \nonumber \\
& \ +  \tau  \sum_{i=1}^N
\int_{\Omega} (\nabla^{\delta_0} u_{i \varepsilon} \cdot \nabla \theta_{\varepsilon})  
\theta_{\varepsilon} |\nabla \theta_{\varepsilon}|^2 dx \ \ 
( =: I_1  + I_2  + I_3 )  \quad \mbox{ a.e.  on }  [0,T].  \label{4-9} 
\end{align} 
It is easy to obtain 
\begin{align}
I_1  & \leq |\theta_0|_{L^{\infty}(\Omega)} |\theta_{\varepsilon t}|_{L^2(\Omega)} 
                    |\nabla \theta_{\varepsilon}|_{L^4(\Omega)}^2  \nonumber  \\                    
 & \leq \frac{\kappa}{2}  |\nabla \theta_{\varepsilon}|_{L^4(\Omega)}^4
  + \frac{1}{2 \kappa}  |\theta_0|_{L^{\infty}(\Omega)}^2  |\theta_{\varepsilon t}|_{L^2(\Omega)}^2,   
               \label{4-10} 
\end{align} 
\begin{align}
I_2  & \leq  2 \kappa |\theta_0|_{L^{\infty}(\Omega)}  \sum_{i,j = 1}^3
        \int_{\Omega}  |\frac{\partial^2  \theta_{\varepsilon}}{\partial x_i \partial x_j}| 
 |\nabla  \theta_{\varepsilon}|^2   dx \nonumber \\
   & \leq  \frac{\kappa}{4}  |\nabla \theta_{\varepsilon}|_{L^4(\Omega)}^4 
 + 72 \kappa |\theta_0|_{L^{\infty}(\Omega)}^2  \sum_{i,j = 1}^3
        \int_{\Omega}  |\frac{\partial^2  \theta_{\varepsilon}}{\partial x_i \partial x_j}|^2  dx, 
    \label{4-11}
\end{align} 
\begin{align}
I_3  & \leq  \tau |\theta_0|_{L^{\infty}(\Omega)}  \sum_{i = 1}^N
        \int_{\Omega}  |\nabla^{\delta_0} u_{i \varepsilon}| |\nabla \theta_{\varepsilon}|^3 dx 
 \nonumber \\
&  \leq  \frac{\kappa}{4}  |\nabla \theta_{\varepsilon}|_{L^4(\Omega)}^4
 + C_{\kappa} |\theta_0|_{L^{\infty}(\Omega)}^4  \int_{\Omega}  |\nabla^{\delta_0} u_{i \varepsilon}|^4 dx
\nonumber \\
&  \leq  \frac{\kappa}{4}  |\nabla \theta_{\varepsilon}|_{L^4(\Omega)}^4
 + C_{\kappa} c_{4,\delta_0} |\theta_0|_{L^{\infty}(\Omega)}^4 
	           |u_{i \varepsilon}|_{L^2(\Omega)}^4  \quad \mbox{ a.e.  on }  [0,T], 
    \label{4-12}
\end{align} 
where $C_{\kappa}$ is a positive constant depending only on $\kappa$. 
From (\ref{4-9}) $\sim$ (\ref{4-12}) it follows that 
\begin{align}
 \frac{\kappa}{8} \int_{\Omega} |\nabla \theta_{\varepsilon}|^4 dx 
\leq &  \frac{1}{2\kappa}   |\theta_0|_{L^{\infty}(\Omega)}^2  |\theta_{\varepsilon t}|_{L^2(\Omega)}^2
 + 72 \kappa |\theta_0|_{L^{\infty}(\Omega)}^2  |\theta_{\varepsilon}|_{H^2(\Omega)}^2  \nonumber \\
 & +   C_{\kappa} c_{4,\delta_0} |\theta_0|_{L^{\infty}(\Omega)}^4 
	           |u_{i \varepsilon}|_{L^2(\Omega)}^4  \quad \mbox{ a.e.  on }  [0,T]. \nonumber
\end{align} 
Hence, thanks to Lemmas \ref{lem4-2} and \ref{lem4-3}  we have proved the conclusion of this lemma. 
\end{proof}

\begin{lemm} \label{lem4-5}
If  the same assumptions as in Lemma \ref{lem4-3}  hold, 
then the set $\{u_{\varepsilon} | \varepsilon \in (0,1] \}$ is bounded in 
$W^{1,2}(0,T; L^2(\Omega)^N)$, $L^{\infty}(0,T; H^1(\Omega)^N)$ and $L^{2}(0,T; H^2(\Omega)^N)$,
and  $\{ \nabla \theta_{\varepsilon} | \varepsilon \in (0,1] \}$ is bounded in
 $L^4(Q(T))^3$ for any $T > 0$. 
\end{lemm}
\begin{proof}
For $\varepsilon \in (0,1]$ and $T > 0$  let $\{\theta_{\varepsilon}, u_{\varepsilon}\}$ be a solution of 
 P$_{\varepsilon}$ on $[0,T]$ and $u_{\varepsilon} = (u_{1\varepsilon}, \cdots, u_{N\varepsilon})$. 
First, we multiply (\ref{eqa2}) by $u_{i \varepsilon t}$ and then get 
\begin{align}
&   |u_{i \varepsilon t }|_{L^2(\Omega)}^2 + 
         \frac{\kappa_i}{2} \frac{d}{dt} |\nabla u_{i\varepsilon}|_{L^2(\Omega)}^2   \nonumber  \\
 = & \ \tau_i \int_{\Omega} (\nabla^{\varepsilon} \theta_{\varepsilon} \cdot \nabla u_{i \varepsilon})  
u_{i \varepsilon t}  dx  
 + \int_{\Omega} R_i(u_{\varepsilon}) u_{i \varepsilon t}  dx   \nonumber  \\
\leq & \frac{1}{2}  |u_{i \varepsilon t }|_{L^2(\Omega)}^2 + 
  \tau_i^2 \int_{\Omega} |\nabla^{\varepsilon} \theta_{\varepsilon}|^2 |\nabla u_{i \varepsilon}|^2  
  dx 
+   \int_{\Omega} |R_i(u_{\varepsilon})|^2  dx  
\quad \mbox{ a.e.  on }  [0,T] \mbox{ for } i. \nonumber 
\end{align} 
Here, for  an arbitrary positive number $\eta$ we can easily get  
\begin{align}
&  \frac{1}{2}  |u_{i \varepsilon t }|_{L^2(\Omega)}^2 + 
         \frac{\kappa_i}{2} \frac{d}{dt} |\nabla u_{i\varepsilon}|_{L^2(\Omega)}^2   \nonumber  \\
\leq & \ \eta    |\nabla u_{i \varepsilon}|_{L^4(\Omega)}^4 
 + C_{\eta} \tau_i^4  |\nabla^{\varepsilon} \theta_{\varepsilon}|_{L^4(\Omega)}^4 
+  C_R^2 |\Omega|  \nonumber  \\
\leq & \ \eta    |\nabla u_{i \varepsilon}|_{L^4(\Omega)}^4 
 + C_{\eta}  c_2^4 \tau_i^4  |\nabla \theta_{\varepsilon}|_{L^4(\Omega)}^4 
+  C_R^2 |\Omega|   \quad \mbox{ a.e.  on }  [0,T] \mbox{ for } i,  \label{4-13} 
\end{align} 
where $C_{\eta}$ is a positive constant depending only on $\eta$ and $C_R$ is already defined in the proof of Lemma \ref{lem4-3}. 

Next, by multiplying  (\ref{eqa2}) by $- \Delta u_{i \varepsilon}$  we can see that 
\begin{align}
&  \frac{1}{2} \frac{d}{dt}   |\nabla u_{i \varepsilon}|_{L^2(\Omega)}^2 + 
       \kappa_i  |\Delta u_{i\varepsilon}|_{L^2(\Omega)}^2   \nonumber  \\
 \leq & \ \tau_i \int_{\Omega} |\nabla^{\varepsilon} \theta_{\varepsilon} | |\nabla u_{i \varepsilon} | 
  |\Delta u_{i \varepsilon}|   
 + \int_{\Omega} |R_i(u_{\varepsilon})|  |\Delta u_{i \varepsilon}|  dx   \nonumber  \\
 \leq & \frac{\kappa_i}{2}  |\Delta u_{i \varepsilon}|_{L^2(\Omega)}^2 
+ \frac{\tau^2}{\kappa_i} \int_{\Omega}  |\nabla^{\varepsilon} \theta_{\varepsilon}|^2 
 |\nabla u_{i \varepsilon} |^2 dx 
+ \frac{C_R^2}{\kappa_i} |\Omega|   \quad \mbox{ a.e.  on }  [0,T] \mbox{ for } i. \nonumber 
\end{align} 
Similarly to (\ref{4-13}), we have 
\begin{align}
&  \frac{1}{2} \frac{d}{dt}   |\nabla u_{i \varepsilon}|_{L^2(\Omega)}^2 + 
      \frac{\kappa_i}{2}  |\Delta u_{i\varepsilon}|_{L^2(\Omega)}^2   \nonumber  \\
  \leq & \eta  |\nabla u_{i \varepsilon}|_{L^4(\Omega)}^4
  +  \frac{\tau^4 c_2^2 C_{\eta}}{\kappa_i^2} \int_{\Omega} 
   |\nabla \theta_{\varepsilon}|^4 dx 
+ \frac{C_R^2}{\kappa_i} |\Omega|   
\quad \mbox{ a.e.  on }  [0,T] \mbox{ for } i. \label{4-14} 
\end{align} 

Moreover, we multiply (\ref{eqa2}) by $u_{i\varepsilon} |\nabla u_{i\varepsilon}|^2$ and 
in the similar way to that of (\ref{4-9}) we observe that 
\begin{align}
& \kappa_i \int_{\Omega} |\nabla u_{i\varepsilon}|^4 dx   
\nonumber  \\
= & \ - \int_{\Omega}  u_{i\varepsilon t} u_{i\varepsilon}  |\nabla u_{i\varepsilon}|^2 dx 
 - \kappa_i \int_{\Omega}  u_{i\varepsilon}  \nabla u_{i\varepsilon} \cdot
              \nabla( |\nabla u_{i\varepsilon}|^2)  dx  + \tau_i \int_{\Omega} (\nabla^{\varepsilon} \theta_{\varepsilon} \cdot \nabla  u_{i\varepsilon})
            ( u_{i\varepsilon}  |\nabla u_{i\varepsilon}|^2)  dx  \nonumber  \\
&  +  \int_{\Omega}  R_i(u_{\varepsilon})  u_{i\varepsilon}  |\nabla u_{i\varepsilon}|^2  dx \ \ 
(=: J_1 +J_2 + J_3 + J_4) 
\quad \mbox{ a.e.  on }  [0,T] \mbox{ for } i.  \label{4-14a}
\end{align} 
By elementary calculations we infer that 
\begin{align}
J_1 & \leq |u_{i\varepsilon}|_{L^{\infty}(\Omega)} \int_{\Omega} | u_{i\varepsilon t}| 
               |\nabla u_{i\varepsilon}|^2dx  \nonumber  \\
&  \leq \frac{\kappa_i}{4}  |\nabla u_{i\varepsilon}|_{L^4(\Omega)}^4 
  + \frac{C_*^2}{\kappa_i}  |u_{i\varepsilon t}|_{L^{2}(\Omega)}^2,   \label{4-15}
\end{align} 

\begin{align}
J_2 & \leq  \kappa_i |u_{i\varepsilon}|_{L^{\infty}(\Omega)} \int_{\Omega} |\nabla u_{i\varepsilon t}| 
               |\nabla (|\nabla u_{i\varepsilon}|^2)| dx  \nonumber  \\
 & \leq 2\kappa_i C_*
   \int_{\Omega} |\nabla u_{i\varepsilon}|^2   |\sum_{i, j=1}^3 
 \frac{\partial^2 u_{i\varepsilon}}{\partial x_i \partial x_j} |^2 dx  \nonumber  \\
&  \leq \frac{\kappa_i}{4}  |\nabla u_{i\varepsilon}|_{L^4(\Omega)}^4 
+ 36 \kappa_i C_*^2 |u_{i\varepsilon}|_{H^2(\Omega)}^2,  \label{4-16}
\end{align} 
\begin{align}
J_3 & \leq  \tau_i |u_{i\varepsilon}|_{L^{\infty}(\Omega)} 
   \int_{\Omega} |\nabla^{\varepsilon}  \theta_{i\varepsilon}| |\nabla u_{i\varepsilon}|^3 dx  
              \nonumber  \\
&  \leq \frac{\kappa_i}{8}  |\nabla u_{i\varepsilon}|_{L^4(\Omega)}^4 
  + C_*^4 C_{\kappa_i, \tau_i}  |\nabla^{\varepsilon} \theta_{\varepsilon}|_{L^{4}(\Omega)}^4
 \nonumber  \\
&  \leq \frac{\kappa_i}{8}  |\nabla u_{i\varepsilon}|_{L^4(\Omega)}^4 
  + c_4 C_*^4 C_{\kappa_i, \tau_i}  |\nabla \theta_{\varepsilon}|_{L^{4}(\Omega)}^4,   \label{4-17}
\end{align} 

\begin{align}
J_4 & \leq |u_{i\varepsilon}|_{L^{\infty}(\Omega)} \int_{\Omega} |R_i(u_{\varepsilon})|  
               |\nabla u_{i\varepsilon}|^2dx  \nonumber  \\
&  \leq \frac{\kappa_i}{8}  |\nabla u_{i\varepsilon}|_{L^4(\Omega)}^4 
  + \frac{2}{\kappa_i}  C_*^2 C_R^2 |\Omega|
\quad \mbox{ a.e.  on }  [0,T] \mbox{ for } i,   \label{4-18}
\end{align} 
where $C_{\kappa_i, \tau_i}$ is a positive constant depending only on $\kappa_i$ and $\tau_i$ for each $i$ and  $C_*$  is defined in (\ref{4-7}).  From (\ref{4-14a}) $\sim$ (\ref{4-18}) it follows that 
 \begin{align}
& \frac{\kappa_i}{4}  \int_{\Omega} |\nabla u_{i\varepsilon}|^4 dx   \nonumber \\
\leq & \frac{C_*^2}{\kappa_i}  |u_{i\varepsilon t}|_{L^{2}(\Omega)}^2 
+  36 \kappa_i C_*^2 |u_{i\varepsilon}|_{H^2(\Omega)}^2
+ c_4 C_*^4 C_{\kappa_i, \tau_i}  |\nabla \theta_{\varepsilon}|_{L^{4}(\Omega)}^4  \label{4-20} \\
& + \frac{2}{\kappa_i}  C_*^2 C_R^2 |\Omega| \mbox{ a.e.  on } [0,T] \mbox{ for } i.  \nonumber 
\end{align} 

Furthermore, by adding (\ref{4-13}) and  (\ref{4-14}), and applying  (\ref{4-20}) and (\ref{LU}) 
we see that 
\begin{align}
& \frac{d}{dt} \sum_{i=1}^N (\frac{1}{2} + \frac{\kappa_i}{2}) |\nabla u_{i\varepsilon}|_{L^{2}(\Omega)}^2
 + \frac{1}{2}  |u_{\varepsilon t}|_{L^2(\Omega)}^2
+ \sum_{i=1}^N  \frac{\kappa_i}{2}  |\Delta u_{i\varepsilon}|_{L^{2}(\Omega)}^2 \nonumber \\
\leq & 2\eta \sum_{i=1}^N  |\nabla u_{i \varepsilon}|_{L^4(\Omega)}^4 
+ (C_{\eta}  c_2^4 \tau_i^4 +  \frac{\tau^4 c_2^2 C_{\eta}}{\kappa_i^2}) 
          |\nabla \theta_{\varepsilon}|_{L^4(\Omega)}^4 
+  (C_R^2 + \frac{C_R^2}{\kappa_i} ) |\Omega|   \nonumber \\
\leq &  2\eta \sum_{i=1}^N (\frac{C_*^2}{\kappa_i}  |u_{i\varepsilon t}|_{L^{2}(\Omega)}^2 
+  36 \kappa_i C_*^2 C_{\Omega}^2 |u_{i\varepsilon}|_{H^2(\Omega)}^2
+ c_4 C_*^4 C_{\kappa_i, \tau_i}  |\nabla \theta_{\varepsilon}|_{L^{4}(\Omega)}^4
+ \frac{2}{\kappa_i}  C_*^2 C_R^2 |\Omega|)  \nonumber \\
& + (C_{\eta}  c_2^4 \tau_i^4 +  \frac{\tau^4 c_2^2 C_{\eta}}{\kappa_i^2}) 
          |\nabla \theta_{\varepsilon}|_{L^4(\Omega)}^4 
+  (C_R^2 + \frac{C_R^2}{\kappa_i} ) |\Omega|  \quad \mbox{ a.e.  on }  [0,T]. \nonumber 
\end{align} 
Since we can take sufficiently small $\eta$ in the above inequality, there exists a positive constant $C_2$ independent of $\varepsilon \in (0,1]$ such that 
\begin{align}
& \frac{d}{dt} \sum_{i=1}^N (\frac{1}{2} + \frac{\kappa_i}{2}) |\nabla u_{i\varepsilon}|_{L^{2}(\Omega)}^2
 +  \frac{1}{2}  |u_{\varepsilon t}|_{L^2(\Omega)}^2
+   \frac{\mu}{4} \sum_{i=1}^N   |\Delta u_{i\varepsilon}|_{L^{2}(\Omega)}^2 \nonumber \\
\leq & C_2(   |\nabla \theta_{\varepsilon}|_{L^4(\Omega)}^4  +1)  \quad \mbox{ a.e.  on }  [0,T]. 
\end{align} 
where $\mu = \min\{\kappa_i| i = 1, 2, \cdots, N\}$. 
Therefore,  the assertion of this lemma is a direct consequence of  Lemma \ref{lem4-4}. 
\end{proof}

\section{Proofs of Theorems \ref{th1} and \ref{th2}} \label{final}

\begin{proof}[Proof of Theorem \ref{th1}]
Let $\delta_0 > 0$ be fixed and for $\varepsilon > 0$ $\{\theta^{(\varepsilon)}, u^{(\varepsilon)} \}$ be a solution of P$_{\varepsilon}$ on $[0,T]$. By Lemmas \ref{lem4-1} $\sim$ \ref{lem4-5} the sets 
$\{\theta^{(\varepsilon)}| \varepsilon \in (0,1] \}$ and $\{u_i^{(\varepsilon)}| \varepsilon \in (0,1] \}$ 
($ i= 1, \cdots, N)$ are bounded in 
$L^{\infty}(Q(T))$, $W^{1,2}(0,T; L^2(\Omega))$, $L^{\infty}(0,T; H^1(\Omega))$, 
$L^2(0,T; H^2(\Omega))$ and $L^4(0,T; W^{1,4}(\Omega))$. 
Then, there exists a subsequence $\{\varepsilon_j\}$ such that 
$\theta^{(j)} := \theta^{(\varepsilon_j)} \to \theta$ and $u^{(j)}_i := u_i^{(\varepsilon_j)} \to u_i$ 
in $L^2(0,T; L^2(\Omega))$, weakly in $W^{1,2}(0,T; L^2(\Omega))$, 
$L^2(0,T; H^2(\Omega))$ and in $L^4(0,T; W^{1,4}(\Omega))$, and 
weakly* in $L^{\infty}(Q(T))$ and $L^{\infty}(0,T; H^1(\Omega))$ as $j \to \infty$, 
where $\theta, u_i \in X(T) \cap  L^4(0,T; W^{1,4}(\Omega))$ for $i = 1,2, \cdots, N$.    

First, we show that 
\begin{equation}
\nabla^{\delta_0} u_i^{(j)} \cdot \nabla \theta^{(j)} \to 
\nabla^{\delta_0} u_i \cdot \nabla \theta \mbox{ weakly in  } L^2(Q(T)) \mbox{ as } j \to \infty
  \mbox{ for each }i. 
\label{5.1} 
\end{equation}
In fact, for each $i$ and  any $\eta \in L^2(Q(T))$ we have 
\begin{align*}
 & |\int_{Q(T))} (\nabla^{\delta_0} u_i^{(j)} \cdot \nabla \theta^{(j)} - 
\nabla^{\delta_0} u_i \cdot \nabla \theta) \eta dx dt| \\
\leq &  |\int_{Q(T))} (\nabla^{\delta_0} u_i^{(j)}  -  \nabla^{\delta_0} u_i) \cdot \nabla \theta^{(j)}  \eta dx dt| 
+ |\int_{Q(T))} \nabla^{\delta_0} u_i  \cdot (\nabla \theta^{(j)} - \nabla \theta) \eta dx dt| \\
=: & I_{1j} + I_{j2}  \quad \mbox{ for each } j. 
\end{align*} 
Since $\nabla^{\delta_0} u_i \in L^{\infty}(Q(T))^3$, namely, $\eta \nabla^{\delta_0} u_i  \in L^2(Q(T))^3$, 
it is easy to see that $I_{2j} \to 0$ as $j \to \infty$.  Also, by (\ref{moli1}) we have 
\begin{align*} 
I_{1j} \leq & (\int_0^T |\nabla^{\delta_0} (u_i^{(j)}  -  u_i) |_{L^4(\Omega)}^4 dt)^{1/4} 
  |\nabla \theta^{(j)}|_{L^4(Q(T))} |\eta|_{L^2(Q(T))}\\
 \leq &  c_{4,\delta_0} (\int_0^T |u_i^{(j)}  -  u_i |_{L^2(\Omega)}^4 dt)^{1/4} 
|\nabla \theta^{(j)}|_{L^4(Q(T))} |\eta|_{L^2(Q(T))}  
 \to 0 \mbox{ as } j \to \infty. 
\end{align*}
Thus  (\ref{5.1}) holds. 

As a next step, we prove 
\begin{equation}
\nabla \theta^{(j)} \to \nabla \theta \mbox{ in  } L^2(Q(T))^3 \mbox{ as } j \to \infty. 
\label{5.10} 
\end{equation}
Indeed, for $j$ it is holds that 
\begin{equation} 
\theta^{(j)}_t - \kappa \Delta \theta^{(j)} 
 - \tau \sum_{i=1}^N \nabla^{\delta_0} \hat{u}_i^{(j)} \cdot \nabla \theta^{(j)}  =0 
                   \mbox{ in } Q(T).  \nonumber 
\end{equation}
The convergences as above and (\ref{5.1}) imply that 
\begin{equation} 
\theta_t - \kappa \Delta \theta 
 - \tau \sum_{i=1}^N \nabla^{\delta_0} u_i  \cdot \nabla \theta  =0 
                   \mbox{ in } Q(T).  \nonumber  
\end{equation}
Accordingly,  we see that    
\begin{equation}
  (\theta^{(j)} - \theta)_t - \kappa \Delta (\theta^{(j)} - \theta)  
 =   \tau \sum_{i=1}^N (\nabla^{\delta_0} u_i^{(j)} \cdot \nabla \theta^{(j)}  - 
        \nabla^{\delta_0} u_i \cdot \nabla \theta)  
                   \mbox{ in } Q(T) \mbox{ for } j.  \label{5.3}
\end{equation}
and multiply (\ref{5.3}) by $\theta^{(j)} - \theta$. Then we obtain 
\begin{align*}
& \frac{1}{2} \frac{d}{dt} |\theta^{(j)} - \theta|_{L^2(\Omega)}^2 
  + \kappa |\nabla (\theta^{(j)} - \theta)|_{L^2(\Omega)}^2 \\
= &  \tau \sum_{i=1}^N \int_{\Omega}  (\nabla^{\delta_0} u_i^{(j)} \cdot \nabla \theta^{(j)}  - 
        \nabla^{\delta_0} u_i \cdot \nabla \theta)  (\theta^{(j)} - \theta) dx
   \mbox{ a.e on } [0,T] \mbox{ for } j  
\end{align*}
and 
\begin{align*}
& \frac{1}{2}  |\theta^{(j)}(T)  - \theta(T)|_{L^2(\Omega)}^2 
  + \kappa \int_0^T |\nabla (\theta^{(j)} - \theta)|_{L^2(\Omega)}^2 dt \\
= &  \tau \sum_{i=1}^N \int_{Q(T)}  (\nabla^{\delta_0} u_i^{(j)} \cdot \nabla \theta^{(j)}  - 
        \nabla^{\delta_0} u_i \cdot \nabla \theta) 
              (\theta^{(j)} - \theta) dx dt   \\
\leq & \tau \sum_{i=1}^N ( |\nabla^{\delta_0} u_i^{(j)}|_{L^4(Q(T))} |\nabla \theta^{(j)}|_{L^4(Q(T))}
 +  |\nabla^{\delta_0} u_i|_{L^4(Q(T))} |\nabla \theta|_{L^4(Q(T))}) 
          |\theta^{(j)} - \theta|_{L^2(Q(T))}  \mbox{ for } j.   
\end{align*}
From this inequality it follows (\ref{5.10}).

The third step of this proof is to get 
\begin{equation}
\nabla^{\varepsilon_j} \theta^{(j)} \cdot \nabla u_i^{(j)} \to \nabla \theta\cdot \nabla u_i
\mbox{ weakly in } L^2(Q(T)) \mbox{ as } j \to \infty \mbox{ for each } i. 
\label{5.4}
\end{equation}
On account of the boundedness for approximate solutions it is sufficient to show 
that 
$$
 \int_{Q(T)} ( \nabla^{\varepsilon_j} \theta^{(j)} \cdot \nabla u_i^{(j)} - \nabla \theta\cdot \nabla u_i) \xi dx dt \to 0 \mbox{ as } j \to \infty 
\mbox{ for } \xi  \in  C^{\infty}(\overline{Q(T)}). $$
Then for  $\xi \in C^{\infty}(\overline{Q(T)})$ and each $j$ it holds that 
\begin{align*}
& |\int_{Q(T)} ( \nabla^{\varepsilon_j} \theta^{(j)} \cdot \nabla u_i^{(j)} - \nabla \theta\cdot \nabla u_i)
 \xi dx dt| \\
\leq & 
 |\int_{Q(T)} ( J_{\varepsilon_j} \ast  (\nabla \theta^{(j)} -  \nabla \theta)  \cdot \nabla u_i^{(j)} \xi dx dt| 
+  |\int_{Q(T)} ( J_{\varepsilon_j} \ast \nabla  \theta  -  \nabla \theta)  \cdot \nabla u_i^{(j)} \xi dx dt| 
\\
& +  |\int_{Q(T)}  \nabla \theta \cdot ( \nabla u_i^{(j)} -  \nabla u_i)  \xi dx dt| 
=:  \hat{I}_{1j} + \hat{I}_{2j} + \hat{I}_{3j}. 
\end{align*} 
Immediately,  for each $j$ we have 
\begin{align*}
\hat{I}_{3j} = & \int_{Q(T)}  \nabla \theta \cdot ( \nabla u_i^{(j)} -  \nabla u_i) \xi dx dt \to 0
\mbox{ as } j \to \infty, 
\end{align*}
since $\xi (\nabla \theta)  \in L^2(Q(T))^3$ and $ \nabla u_i^{(j)} \to  \nabla u_i$ weakly in 
$L^2(0,T; Q(T))^3$. Also, because $\nabla \theta \in L^4(Q(T))^3$,  it is easy to see 
$$ \hat{I}_{2j} \leq  |J_{\varepsilon_j} \ast \nabla \theta  -  \nabla \theta|_{L^4(Q(T))}
        |\nabla u_i^{(j)}|_{L^4(Q(T))} |\xi|_{L^2(Q(T))} \to 0 \mbox{ as } j \to \infty. 
$$ 
Here, by (\ref{moli3}) we infer that  
\begin{align*} 
\hat{I}_{1j} \leq & 
 |J_{\varepsilon_j} \ast  (\nabla \theta^{(j)} -  \nabla \theta)|_{L^2(Q(T))} 
    |\nabla u_i^{(j)}|_{L^4(Q(T))}  |\xi|_{L^4(Q(T))} \\
 \leq & 
 |\nabla \theta^{(j)} -  \nabla \theta|_{L^2(Q(T))} 
    |\nabla u_i^{(j)}|_{L^4(Q(T))}  |\xi|_{L^4(Q(T))} \to 0 \mbox{ as } j \to \infty. 
\end{align*} 
The above arguments guarantee (\ref{5.4}). 

From (\ref{5.4}) it follows 
$$ u_{it} - \kappa_i \Delta u_i - \tau_i \nabla \theta \cdot \nabla u_i = R_i(u) 
 \quad \mbox{ a.e. on } Q(T) \mbox{ for } i. 
$$
Therefore. $\{\theta, u\}$ is a solution of $P$ on $[0,T]$ for any $T  > 0$ and 
then is also a solution of P on $[0,\infty)$. 

(\ref{assert2})  is a direct consequence of Lemma \ref{lem4-2}. 
Thus we have proved Theorem \ref{th1}. 
\end{proof}

\begin{proof}[Proof of Theorem \ref{th2}]
First, we shall prove (1). 
By Theorem \ref{th1} there exists a solution $\{\theta, u\}$ of P on $[0,T]$. 
Since $\theta_0 \in W^{1,\infty}(\Omega)$ and 
 $\nabla^{\delta_0}u_i \in L^{\infty}(Q(T))^3$ in (\ref{eq1}) for $i$, where 
$u = (u_1, \cdots, u_N)$, 
the classical theory for linear parabolic equations, for instance \cite{LSU}, guarantees that 
$\nabla \theta \in L^{\infty}(Q(T))^3$. Similarly, we can show that 
$\nabla u_i  \in L^{\infty}(Q(T))^3$ for $i$, because $R_i(u) \in L^{\infty}(Q(T))$. 

Next,  let $\{\theta^{(k)}, u^{(k)}\}$ be a solution of P for $k = 1, 2$ satisfying 
$\theta^{(k)} \in L^{\infty}(0,T; W^{1,\infty}(\Omega))$ and 
$u_i^{(k)} \in L^{\infty}(0,T; W^{1,\infty}(\Omega))$, $k = 1, 2$ and each $i$. 
We put $\theta = \theta^{(1)} - \theta^{(2)}$,  $u = u^{(1)} - u^{(2)}$,
$u^{(k)} = (u_1^{(k)}, \cdots, u_N^{(k)})$, $k = 1, 2$. 
Since $u^{(k)} \in L^{\infty}(Q(T))^N$ for $k$, we can take a positive constant 
$C_R'$ such that 
$$ |R_i(u^{(1)})  - R_i(u^{(2)})| \leq C_R' |u| \mbox{ a.e. on } Q(T). $$

First, it is  straightforward to see that 
\begin{align}
& \theta_t - \kappa \Delta \theta = \tau \sum_{i=1}^N
    ( \nabla^{\delta_0} u_i^{(1)} \cdot \nabla \theta^{(1)} 
       -  \nabla^{\delta_0} u_i^{(2)} \cdot \nabla \theta^{(2)}) \quad \mbox{ in } Q(T), \label{5.5}
\\
& u_{it} - \kappa_i \Delta u_i = \tau_i ( \nabla u_i^{(1)} \cdot \nabla \theta^{(1)} 
       -  \nabla u_i^{(2)} \cdot \nabla \theta^{(2)} )
    + R_i(u^{(1)})  - R_i(u^{(2)})  \quad \mbox{ in } Q(T). \label{5.6}
 \end{align}
We multiply (\ref{5.5}) by $\theta$ and integrate it over $\Omega$. Then we see that 
\begin{align}
& \frac{1}{2} \frac{d}{dt} |\theta|_{L^2(\Omega)}^2 + \kappa |\nabla \theta|_{L^2(\Omega)}^2 
        \nonumber \\
= & \tau \sum_{i=1}^N \int_{\Omega}
     ( \nabla^{\delta_0} u_i  \cdot \nabla \theta^{(1)} 
      +  \nabla^{\delta_0} u_i^{(2)} \cdot \nabla \theta) \theta dx  \nonumber \\
\leq &  \tau_i \sum_{i=1}^N (
        |\nabla \theta^{(1)}|_{L^{\infty}(Q(T))} |\nabla^{\delta_0} u_i|_{L^2(\Omega)} |\theta|_{L^2(\Omega)}
  +   |\nabla^{\delta_0} u_i^{(2)}|_{L^{\infty}(Q(T))} |\nabla \theta|_{L^2(\Omega)} |\theta|_{L^2(\Omega)}) 
 \nonumber \\
\leq &  C (\sum_{i=1}^N  |u_i|_{L^2(\Omega)}^2 +  |\theta|_{L^2(\Omega)}^2)
   + \frac{\kappa}{4} |\nabla \theta|_{L^2(\Omega)}^2 \mbox{ a.e. on } [0,T],   \label{5.7} 
\end{align}
where $C$ is a positive constant. 

Similarly, from (\ref{5.6}) we observe that 
\begin{align}
& \frac{1}{2} \frac{d}{dt} |u_i|_{L^2(\Omega)}^2 + \kappa_i |\nabla u_i|_{L^2(\Omega)}^2 
        \nonumber \\
= & \tau_i \int_{\Omega}
     ( \nabla u_i  \cdot \nabla \theta^{(1)} 
      +  \nabla u_i^{(2)} \cdot \nabla \theta) u_i dx  
     + \int_{\Omega} (R_i(u^{(1)}) -  R_i(u^{(2)}) )  u_i dx 
\nonumber \\
\leq & \tau_i  
     (|\nabla \theta^{(1)}|_{L^{\infty}(Q(T))} |\nabla u_i|_{L^2(\Omega)} 
      +  |\nabla u_i^{(2)}|_{L^{\infty}(Q(T))}  |\nabla \theta|_{L^2(\Omega)} 
     +  C_R' |u|_{L^2(\Omega)} )  |u_i|_{L^2(\Omega)}   
\nonumber \\
\leq &  \frac{\kappa}{4}  |\nabla \theta|_{L^2(\Omega)}^2
+ \frac{\kappa_i}{2}  |\nabla u_i|_{L^2(\Omega)}^2  
+  C'( |u_i|_{L^2(\Omega)}^2 +  |u|_{L^2(\Omega)}^2) 
\mbox{ a.e. on } [0,T] \mbox{ for }  i,   \label{5.8} 
\end{align}
where $C'$ is a positive constant. 

By adding (\ref{5.7}) and (\ref{5.8}), we infer that 
\begin{align*}
& \frac{1}{2} \frac{d}{dt}( |\theta|_{L^2(\Omega)}^2 +  |u|_{L^2(\Omega)}^2) 
+ \kappa |\nabla \theta|_{L^2(\Omega)}^2 
+ \sum_{i=1}^N  \kappa_i |\nabla u_i|_{L^2(\Omega)}^2 \\
\leq & \hat{C}  (|\theta|_{L^2(\Omega)}^2 +  |u|_{L^2(\Omega)}^2) \mbox{ a.e. on } [0,T],  
\end{align*}
where $\hat{C}$ is a positive constant. 
Therefore,  by applying Gronwall's inequality  to the above inequality, 
we conclude that $\theta = 0$ and $u = 0$ a.e on $Q(T)$. 
\end{proof} 

\bibliographystyle{plain}

\end{document}